\documentclass[final]{siamltex}
\usepackage{amsfonts,amsmath,amssymb}
\usepackage{mathrsfs}
\usepackage{enumerate}
\usepackage{xspace,mydef,boldfonts}
\usepackage{hyperref}
\usepackage{esint}
\usepackage{graphicx}

\newtheorem{remark}[theorem]{Remark}
\numberwithin{equation}{section}

\title{Finite element approximation of the Isaacs equation\thanks{AJS has been partially supported by NSF grant DMS-1418784. WZ has been partially supported by the Brin Postodctoral Fellowship of the University of Maryland.}}

\author{Abner J.~Salgado\thanks{Department of Mathematics, University of Tennessee, Knoxville, TN 37996, USA.
\texttt{asalgad1@utk.edu}}
\and
Wujun Zhang\thanks{Department of Mathematics, University of Maryland, College Park, MD 20742, USA.
\texttt{wujun@umd.edu}}
}

\begin{document}
\maketitle

\begin{abstract}
  We propose and analyze a two-scale finite element method for the Isaacs equation. The fine scale is given by the mesh size $h$ whereas the coarse scale $\varepsilon$ is dictated by an integro-differential approximation of the partial differential equation. We show that the method satisfies the discrete maximum principle provided that the mesh is weakly acute. This, in conjunction with weak operator consistency of the finite element method, allows us to establish convergence of the numerical solution to the viscosity solution as $\varepsilon, h\to0$, and $\varepsilon \gtrsim h^{1/2}|\log h|$. In addition, using a discrete Alexandrov Bakelman Pucci estimate we deduce rates of convergence, under suitable smoothness assumptions on the exact solution.
\end{abstract}

\begin{keywords}
  Fully nonlinear equations, viscosity solution, discrete maximum principle
\end{keywords}

\begin{AMS}
65N12,     
65N15,                      
65N30,                      
35J60,                      
35D40,                      
35Q91,                      
\end{AMS}

\date{Submitted \today.}

\section{Introduction}
\label{sec:intro}

Fully nonlinear elliptic partial differential equations (PDE) arise naturally from differential geometry, optimal mass transportation, stochastic optimal control and other fields of science and engineering. In spite of their wide range of applications, the numerical methods for this type of PDEs is still under development and this is particularly the case if one wants to apply the finite element method (FEM). A major difficulty in their numerical approximation is that, for fully nonlinear PDEs, the correct notion of solution is the so-called viscosity solution, which is based on the maximum principle instead of a variational one. This is reflected in the fact that, in contrast to an extensive literature for linear and quasilinear elliptic PDEs in divergence form, the numerical approximation via finite elements of fully nonlinear PDEs reduces to a few papers; we refer the reader to \cite{FengGlowinskiNeilan13} for an overview. The situation is somewhat more satisfactory for finite difference approximations, where convergence to the viscosity solution without rates was studied in the early works \cite{BarlesSouganidis91, KuoTrudinger92}. Rates of convergence for convex/concave fully nonlinear elliptic equations have been established, \eg by Krylov \cite{Krylov97, Krylov05}; Barles and Jakobsen \cite{BarlesJakobsen05} and Debrabant and Jakobsen \cite{MR3042570}. However, rates of convergence for the nonconvex/nonconcave case remained an open problem until Caffarelli and Souganidis \cite{CaffarelliSouganidis08} established a rate of convergence for elliptic equations of the form
\[
F[\D^2u](x) = f(x),
\]
with Dirichlet boundary conditions. This result was extendend by Turanova \cite{MR3412399} to the case where $F$ is also dependent on $x$ and by Krylov \cite{MR3355497} to the Isaacs equation, instead of a general nonconvex fully nonlinear PDE, but he allows the operator $F$ to depend on $x$, $\GRAD u$ and $u$.

In all the works mentioned above, convergence to the viscosity solution hinges on operator consistency and monotonicity, which are nontrivial properties to be satisfied, especially in the FEM. To overcome this, the recent work \cite{RHNWZ} introduced a two scale FEM for linear elliptic PDEs in nondivergence form and, on the basis of monotonicity and \emph{weak} operator consistency, the authors were able to prove convergence of the method with rates.
The purpose of this work is to extend these ideas to the case of fully nonlinear elliptic PDEs. We consider, as an example, the Isaacs equation
\begin{equation}
\label{eq:isaacs}
  \frakF[\ue](x) := \inf_{\alpha \in \calA} \sup_{\beta \in \calB} \left[ A^{\alpha,\beta}(x): \D^2 \ue(x) \right] = \fe(x) \ \text{in } \Omega,
  \quad
  \ue = 0 \ \text{on } \partial\Omega,
\end{equation}
where $\Omega \subset \Real^d$ ($d\geq2$) is an open, bounded and convex domain, $\fe \in C^{0,1}(\bar\Omega)$, the sets $\calA$ and $\calB$ are arbitrary finite sets and the matrices $A^{\alpha,\beta} \in C^{0,1}(\bar\Omega,\GL_d(\Real))$ are symmetric and uniformly elliptic, in the sense that there are constants $\lambda,\Lambda \in \Real$, with $0<\lambda \leq \Lambda$, such that
\begin{equation}
\label{eq:ellipticity}
  \lambda I \leq A^{\alpha,\beta} \leq \Lambda I \quad \forall \alpha \in \calA, \beta \in \calB .
\end{equation}

We will introduce, following \cite{RHNWZ}, a two-scale FEM, show its convergence to the viscosity solution of \eqref{eq:isaacs} and provide rates of convergence. 
The method is based on the approximation of \eqref{eq:isaacs} proposed by Caffarelli and Silvestre in \cite{MR2667633}: we formally rewrite
\begin{equation*}
  \frakF[\ue](x) = \frac\lambda2 \LAP \ue(x) 
    + \inf_{\alpha \in \calA} \sup_{\beta \in \calB} \left[ \left( A^{\alpha,\beta}(x) -  \frac\lambda2 I \right): \D^2 \ue(x) \right]  
\end{equation*}
and we approximate the operator above by the \emph{integro-differential} operator
\begin{equation*}
  \frakF^\vare[u](x) := \frac\lambda2 \LAP u(x) 
    + \inf_{\alpha \in \calA} \sup_{\beta \in \calB} \Ie{u}(x) 
\end{equation*}
where 
\begin{equation*}
  \Ie{u}(x) = \frac1{\vare^{d+2} \det M^{\alpha,\beta}(x)} \int_{\Real^d} \frakd u(x,y) 
  \varphi\left( \frac {M^{\alpha,\beta}(x)^{-1} y}{\vare} \right) \diff y,
\end{equation*}
with
\[
  M^{\alpha,\beta}(x) := \left( A^{\alpha,\beta}(x) - \frac\lambda2 I \right)^{1/2}.
\]
Hereafter, $\varphi(y)$ is a radially symmetric smooth function with compact support in the unit ball $B_1$ of $\Real^d$, where $d\geq 1$ is the dimension, that verifies 
\[
  \int_{\mathbb R^d} |y|^2 \varphi(y) \diff y = d
\]
and 
\begin{equation}
\label{eq:defoffrakd}
  \frakd u(x,y):= u(x+y) - 2 u(x) + u (x-y) 
\end{equation}
is the centered second difference operator.
The operator $\Ie{\cdot}$ is a consistent approximation of $\left( A^{\alpha, \beta}(x) - \frac{\lambda}{2} I \right) : \D^2 u(x)$ in the sense that if $u$ is a quadratic polynomial, then
\begin{equation}
  \label{quadratics}
  \Ie{u}(x)  = \left( A^{\alpha, \beta}(x) -  \frac{\lambda}{2}I \right) : \D^2 u (x)  \qquad \forall \vare>0, \quad \forall u \in \polP_2,
\end{equation}
see Lemma~\ref{lem:approximation}.

We discretize \eqref{eq:isaacs} by using this approximation to obtain
\begin{align*}
  \frakF_h^\vare[u^\vare_h](z) &:=  -\frac\lambda2 \int_{\Omega} \GRAD u_h^\vare(x) \cdot \GRAD \phi_{z}(x) \diff x 
  + \inf_{\alpha \in \calA} \sup_{\beta \in \calB} \Ie{u_h^\vare}(z)\int_\Omega \phi_{z}(x) \diff x 
  \\
  &=  \int_\Omega \phi_{z}(x) \fe(x) \diff x \qquad \forall z \in \Nh,
\end{align*}
where $u_h^\vare$ is a continuous and piecewise affine finite element function over a mesh $\T_h$, $\N_h$ is the set of internal nodes of $\T_h$  and $\{\phi_z\}_{z \in \Nh}$ is the Lagrange nodal basis of the finite element space. 
We show that the FEM is monotone provided that meshes are weakly acute. To show existence and uniqueness, we employ a \emph{discrete} version of Perron's method, which seems to not have been considered before, especially in the finite element literature.  Exploiting monotonicity, and using the notion of \emph{weak consistency} introduced by Jensen and Smears in \cite{JensenSmears13} we show convergence to the viscosity solution following \cite{BarlesSouganidis91, KuoTrudinger92}.

We also derive rates of convergence for the method. 
The main difficulty to derive a rate of convergence is to establish a suitable notion of stability for the FEM applied to this fully nonlinear PDE. To address this issue, we resort to the discrete Alexandrov Bakelman Pucci (ABP) estimate of \cite{RHNWZ}, which reads
\begin{equation}
\label{intro:ABP}
  \sup_{\Omega} (u_h^\vare)^- \leq c \left( 
  \sum_{\left\{z \in \Nh: u_h^\vare(z) = \Gamma (u_h^\vare)(z) \right\}} |f_z|^d |\omega_z| \right)^{1/d}.
\end{equation}
Here $\left\{z \in \Nh: u_h^\vare(z) = \Gamma (u_h^\vare)(z) \right\}$ denotes the \emph{(lower) nodal contact set}, defined in \eqref{contactset},
\[
  f_z = \left( \int _\Omega \fe(x) \phi_z(x) \diff x \right) \left( \int_\Omega \phi_z(x) \diff x \right)^{-1},
\]
$|\omega_z|$ stands for the volume of the star $\omega_z= \supp\phi_z$ associated with the node $z \in \Nh$. Note that the nodal contact set is just a finite collection of nodes. With \eqref{intro:ABP} we obtain control of the negative part of $u_h^\vare$. If we consider the concave envelope and corresponding (upper) contact set we can estimate the positive part. A combination of these bounds yields stability, in the $L^\infty$-norm, of $u_h^\vare$ in terms of the $L^d$-norm of the right hand side $\fe$. As it has become now customary in numerical analysis, the correct notions of consistency and stability guarantee convergence. Moreover, by combining them with the regularity estimates for the Isaacs equation of \cite{MR2667633}, we show that, for some $\sigma > 0$,
\[
 \| \ue - u_h^\vare \|_{L^\infty(\Omega)} \lesssim h^{\frac\sigma{\sigma+2}} |\log h |^{\frac\sigma{\sigma+2}} \| \fe \|_{C^{0,1}(\bar\Omega)} .
\]

Finally, we discuss how to practically realize the method in question. We first study, following \cite{MR2551155}, a variant of Howard's algorithm to solve the ensuing discrete (nonlinear) systems. Moreover, we also consider a numerical scheme with quadrature which, provided the quadrature rule is properly chosen, possesses all the aforementioned stability and convergence properties.

The rest of this paper is organized as follows. In Section~\ref{sec:wujun}, we recall the approximation of elliptic problems by integro-differential equations and the main regularity results that follow from it. Section~\ref{sec:fem} presents our discretization and proves convergence to the viscosity solution. We recall the discrete ABP estimate for finite element methods in Section~\ref{sec:rates}. The discrete ABP estimate allows us to show stability and rates of convergence for our discretization. We discuss some implementation details in Section~\ref{sec:implementation}, where we present a convergent iterative scheme and a scheme with quadrature.

We will follow standard notation concerning differential operators and function spaces. The relation $A \lesssim B$ means that there is a nonessential constant $c$ such that $A \leq c B$. The value of this constant might change at each occurrence. By $A \gtrsim B$ we mean $B \lesssim A$.

\section{Approximation of elliptic problems by integro-differential operators}
\label{sec:wujun}

Let us review the approximation, proposed by Caffarelli and Silvestre in \cite{MR2667633}, of fully nonlinear elliptic PDEs by an integro-differential equation and the convergent finite element scheme of \cite{RHNWZ} for the approximation of elliptic problems in nondivergence form based on this idea. In this work, we exploit this approximation to discretize the fully nonlinear problem \eqref{eq:isaacs}.

\subsection{Integral operator}
Let us begin by fixing some notation. Recall that, owing to \eqref{eq:ellipticity}, for all $x\in \Omega$ the matrices $A^{\alpha,\beta}(x)$ are uniformly positive definite. Thus, we can define
\[
  M^{\alpha,\beta}(x) := \left( A^{\alpha,\beta}(x) - \frac\lambda2 I \right)^{1/2}.
\]
Let $\vare>0$ and $\varphi$ be a radially symmetric function with compact support in the unit ball $B_1$ and such that $\int |x|^2 \varphi(x) \diff x = d$. Let $Q=\sqrt{2/\lambda}$ and notice that, again because of \eqref{eq:ellipticity}, for all $\alpha \in \calA$ and $\beta \in \calB$ we have that if
\[
  \varphi_{x,\vare}^{\alpha,\beta}: y \mapsto \varphi\left( \frac{1}{\vare} M^{\alpha,\beta}(x)^{-1} y \right),
\]
then $\supp \varphi_{x,\vare}^{\alpha,\beta} \subset B_{Q\vare}$. For this reason, we define
\begin{equation}
  \Omega_\vare := \Omega \bigcup \left\{ x \in \Real^d: \dist(x,\partial\Omega) \leq Q\vare \right\}.
\end{equation}

Given a function $w \in C(\bar\Omega_\vare)$ we define, for each $\alpha$ and $\beta$, the function $\Ie w$ by
\begin{equation}
\label{eq:defofI}
  \Ie{w}(x) = \frac1{\vare^{d+2} \det M^{\alpha,\beta}(x)} \int_{\Real^d} \frakd w(x,y) 
  \varphi_{x,\vare}^{\alpha,\beta}(y) \diff y,
\end{equation}
where we denoted by $\frakd w(x,y)$ the second order difference of the function $w$ at the point $x$ in the direction $y$ which is given in \eqref{eq:defoffrakd}. Notice that, for $x \in \Omega$, $x \pm y \in \Omega_\vare$.
It is easy to check that if $w$ is a quadratic polynomial, then we have $\frakd w (x,y) = \D^2 w(x) : (y \otimes y)$ for all $x \in \Omega$. 

The integral operator $\Ie{w}$ is a consistent approximation of the differential operator $ \left( A^{\alpha,\beta} - \frac\lambda2 I \right) : \D^2 w$
in the following sense \cite{RHNWZ,MR2667633}.

\begin{lemma}[approximation properties of $\Ie{\cdot}$]
\label{lem:approximation}
Let $\Ie{\cdot}$ be the integral operator defined by \eqref{eq:defofI}. For $\vare > 0$ and $x \in \bar\Omega$, denote $U_{Q_\vare}(x) = \bar{B}_{Q\vare}(x) \cap \bar \Omega$, where $B_{Q\vare}(x)$ is the ball of radius $Q\vare$ and center $x$.
\begin{enumerate}[(a)]
  \item If $w \in \polP_2$, \ie it is a quadratic polynomial, then
  \[
    \Ie{w}(x) = A^{\alpha,\beta}(x):\D^2 w(x).
  \]

  \item If $w \in C^2(\bar\Omega_\vare)$ then, as $\vare \to 0$, we have that
  \[
    \Ie{w}(x) \to \left( A^{\alpha,\beta}(x) - \frac\lambda2 I \right) : \D^2 w (x).
  \]
%
\end{enumerate}
\end{lemma}

Let us, from now on, assume that the matrices $M^{\alpha,\beta}$ have a uniform modulus of continuity $\varpi$. In other words, there is a nondecreasing function $\varpi$ such that
\begin{equation}
\label{eq:modulus}
  \sup_{x_1,x_2 \in \bar\Omega, |x_1 - x_2|\leq t} \| M^{\alpha,\beta}(x_1) - M^{\alpha,\beta}(x_2) \| \leq \varpi(t), 
    \quad \forall \alpha \in \calA,  \beta \in \calB.
\end{equation}
Under this assumption, we can show that, for every $\vare > 0$, the integral operator maps continuously $C^{0,1}(\bar\Omega_\vare)$ into $C^{0}(\bar\Omega)$.

\begin{lemma}[continuity of $\Ie{\cdot}$]
\label{lem:continuous}
If $w \in C^{0,1}(\bar\Omega_\vare)$, then, for all $\alpha \in \calA$ and $\beta \in \calB$ and all $x,z \in \bar\Omega$
\[
  \left| \Ie{w}(x) - \Ie{w}(z) \right| \lesssim \left(\frac{|x-z|}{\vare^2} + \frac{\varpi(|x-z|)}\vare \right) | w |_{C^{0,1}(\bar\Omega_\vare)},
\]
where the hidden constant is independent of $\alpha$, $\beta$, $x$, $z$, $\vare$ and $w$.
\end{lemma}
\begin{proof}
By the radial symmetry of $\varphi$
\[
  \Ie{w}(x) = \frac2{\vare^2} \int_{B_1} \left( w(x + \vare M^{\alpha,\beta}(x) y ) - w(x) \right) \varphi(y) \diff y.
\]
Thus,
\[
  \left| \Ie{w}(x) - \Ie{w}(z) \right| \leq \textrm{I} + \textrm{II},
\]
with
\begin{align*}
  \textrm{I} &= \frac2{\vare^2} \int_{B_1} \left| w(x+\vare M^{\alpha,\beta}(x)y) - w(z+\vare M^{\alpha,\beta}(z)y) \right| \varphi(y) \diff y, \\
  \textrm{II} &= \frac2{\vare^2} \int_{B_1} \left| w(x) - w(z) \right| \varphi(y) \diff y.
\end{align*}
Evidently,
\[
  \textrm{II} \lesssim \frac{|x-z|}{\vare^2} | w|_{C^{0,1}(\bar\Omega)}.
\]
It remains then to estimate the first term. To do so we, again, use that $w \in C^{0,1}(\bar\Omega_\vare)$ to obtain
\[
  \textrm{I} \lesssim \frac{ |w|_{C^{0,1}(\bar\Omega_\vare)}}{\vare^2} 
  \sup_{y \in B_1} \left|x - z + \vare\left( M^{\alpha,\beta}(x) - M^{\alpha,\beta}(z) \right) y \right|.
\]
an application of the triangle inequality, together with \eqref{eq:modulus} imply
\[
  \textrm{I} \lesssim \left( \frac{|x-z|}{\vare^2} + \frac{\varpi(|x-z|)}\vare \right) |w|_{C^{0,1}(\bar\Omega_\vare)},
\]
where the hidden constant is uniform in $\alpha$ and $\beta$.

Gathering the obtained bounds for $\textrm{I}$ and $\textrm{II}$ allows us to conclude.
\end{proof}

In \cite{MR2667633} Caffarelli and Silvestre proposed an approximation of the Isaacs equation \eqref{eq:isaacs} by the following integro-differential problem
\begin{equation}
\label{eq:isaacseps}
  \frakF^\vare[u^\vare](x) := \frac\lambda2 \LAP u^\vare(x) 
  + \inf_{\alpha \in \calA} \sup_{\beta \in \calB} \Ie{u^\vare}(x) = \fe(x) \ \text{in } \Omega, \quad
  u^\vare = 0 \ \text{in } \Omega_\vare\setminus\Omega.
\end{equation}

The approximation property of the integral operator $\Ie\cdot$ given in Lemma~\ref{lem:approximation}, allows us to relate $u^\vare$ and $\ue$, which solve \eqref{eq:isaacseps} and \eqref{eq:isaacs}, respectively, as follows \cite{MR2667633}.

\begin{proposition}[approximation of the Isaacs equation]
\label{prop:CS}
Let $\Omega \subset \Real^d$ be bounded and convex and $\fe \in C^{0,s}(\bar\Omega)$ for some $s\in [0,1)$. There exists a unique function $u^\vare \in C^{2,s}(\Omega)$ that solves \eqref{eq:isaacseps}.  Moreover, if $\fe \in C^{0,1}(\bar\Omega)$, this solution satisfies
\[
  \| u^\vare \|_{C^{1,s}(\Omega)} + \| u^\vare \|_{C^{0,1}(\bar\Omega)} \lesssim \| u^\vare \|_{L^\infty(\Omega)} + \| \fe \|_{L^\infty(\Omega)}.
\]
In addition, there is a $\sigma>0$ such that
\begin{equation}
\label{eq:rateuueps}
  \| \ue - u^\vare \|_{L^\infty(\Omega)} \lesssim \vare^\sigma \| \fe \|_{C^{0,1}(\bar\Omega)},
\end{equation}
where $\ue$ is the (unique) viscosity solution to \eqref{eq:isaacs}.
\end{proposition}

An important consequence of Proposition~\ref{prop:CS} is that, in the case when the coefficient matrices $\{A^{\alpha,\beta}\}$ are independent of $x$, we have the following regularity result for $\ue$.

\begin{corollary}[regularity of $\ue$]
\label{cor:reg}
Assume that, for all $\alpha \in \calA$ and $\beta \in \calB$ $A^{\alpha,\beta}(x) = A^{\alpha,\beta}$, that $\Omega$ satisfies an exterior ball condition and that $\fe \in C^{0,1}(\bar \Omega)$. If $\ue$ is the viscosity solution of \eqref{eq:isaacs} then there is a $s>0$ such that $\ue \in C^{1,s}(\Omega)\cap C^{0,1}(\bar\Omega)$.
\end{corollary}

The importance of this result lies in the fact that $\frakF$ is not convex nor concave and, for that reason, the maximal regularity we can assert for $\ue$ is $C^{1,s}(\Omega)$, for some $s>0$. We refer the reader, for instance, to the works by Nadirashvili and Vl{\u a}du{\c t} \cite{MR2373018,MR2391642}, who have constructed viscosity solutions to nonconvex fully nonlinear elliptic equations whose Hessian is not bounded. This is in sharp contrast with, for instance, the stationary Hamilton Jacobi Bellman equation where it can be shown that the solution belongs to $C^{2,s}(\Omega)$ for some $s>0$. Let us, finally, comment that the rate of convergence given in \eqref{eq:rateuueps} cannot be improved; see the last paragraph of \cite{MR2667633}.

In this paper, based on the idea of an integro-differential approximation, we propose a finite element method for Isaacs equation \eqref{eq:isaacs}.

\section{Finite element discretization and convergence}
\label{sec:fem}
Here we describe the scheme we use to approximate \eqref{eq:isaacs} and show that it converges to the (unique) viscosity solution. 

\subsection{Description of the scheme}
\label{sub:setting}
Let $\T_h$ be a quasi-uniform mesh of size $h$ of the domain $\Omega$. We denote by $\Nh$ and $\Nh^{\partial}$ the collection of interior and boundary nodes of $\T_h$, respectively. Denote by $\polP_1$ the space of polynomials of degree one. We define the finite element spaces
\[
  \polV_h := \left\{ v \in C(\bar\Omega) :  v_{|T} \in \polP_1 \; \forall T \in \T_h \right\}
\]
and
\[
  \polV_h^0 := \left\{ v \in \polV_h : v_{|\partial\Omega} = 0 \right\}.
\]
We denote by $\{\phi_z\}_{z \in \Nh}$ the Lagrange nodal basis of $\polV_h^0$. Extending by zero, we can always assume that functions of $\polV_h^0$ are defined on $\Omega_\vare$. Although not needed for the implementation, we introduce a quasiuniform mesh $\T_h^\vare$ on $\Omega_\vare$ and assume that, on $\Omega$, it coincides with $\T_h$. We denote by $\Nh^\vare$ the nodes of $\T_h^\vare$ that are not in $\bar\Omega$.

For any $w_h \in \polV_h$ and $z \in \Nh$, we define the discrete Laplacian by
\begin{align}\label{discretelaplace}
  \LAP_h w_h(z) =  - \left(\int_\Omega \phi_{z}(x) \diff x\right)^{-1} \int_\Omega \GRAD w_h(x) \cdot \GRAD \phi_z(x) \diff x.
\end{align}
With this notation at hand, we now define our scheme. We seek a function $u_h^\vare \in \polV_h^0$ such that, for every $z \in \Nh$, satisfies
\begin{equation}
\label{eq:isaacsepsh}
  \frakF_h^\vare[u^\vare_h](z) := \frac\lambda2 \LAP_h u_h^\vare(z) 
  + \inf_{\alpha \in \calA} \sup_{\beta \in \calB} \Ie{u_h^\vare}(z)  = f_{z},
\end{equation}
where
\[
  f_{z} = \left( \int_\Omega \phi_{z}(x) \diff x \right)^{-1} \int_\Omega \fe(x) \phi_{z}(x) \diff x.
\]

\subsection{Galerkin projection}

In the analysis that follows we will make repeated use of the Galerkin projection and its properties. The Galerkin projection $\calG_h : W^1_1(\Omega) \cap C(\bar\Omega_\vare \setminus \Omega) \to \polV_h$ is the map that verifies $\calG_h w (z) = w(z)$ for $z \in \Nh^{\partial} \cup \Nh^\vare$ and
\begin{equation}
\label{eq:defofGh}
\int_\Omega \GRAD \calG_h w(x) \SCAL \GRAD v_h(x) \diff x = \int_\Omega \GRAD w(x) \SCAL \GRAD v_h(x) \diff x \quad \forall v_h \in \polV_h^0.
\end{equation}
Notice that, if $w_{|\Omega} \in C^2(\Omega)$, setting $v_h = \phi_{z}$ in \eqref{eq:defofGh} and integrating by parts we get
\[
  \LAP_h \calG_h w(z) = \left( \int_\Omega \phi_{z}(x) \diff x \right)^{-1} \int_\Omega \phi_{z}(x) \LAP w(x) \diff x,
\]
that is, the discrete Laplacian of the Galerkin projection at a node is a weighted average of the Laplacian of the original function around this node. 

It is well known that the Galerkin projection has near optimal approximation properties \cite{SchatzWahlbin82} in the $L^{\infty}$-norm,
\[
\inftynorm{ w - \calG_h w } \lesssim |\log h| \inf_{v_h \in \Vh} \inftynorm{w - v_h}.
\]
Combining this with a standard interpolation estimate over $\bar\Omega_\vare\setminus \Omega$ we conclude that if $w \in C^{0,1}(\bar\Omega_\vare)$, then we have
\begin{equation}
\label{eq:Galproj}
\| w - \calG_h w \|_{L^\infty(\Omega_\vare)} \lesssim h |\log h| \| w \|_{C^{0,1}(\bar\Omega_\vare)}.
\end{equation}
Owing to \eqref{eq:Galproj}, for every $z \in \Nh$ we obtain
\[
  \left| \frakd \calG_h w (z, y) - \frakd w (z, y) \right| \lesssim h |\log h| \| w \|_{C^{0,1}(\bar\Omega_\vare)} .
\]
By definition \eqref{eq:defofI} of the integral operator $\Ie{\cdot}$, we deduce that for all nodes $z \in \Nh$ we have
\begin{equation}
\label{estimateIe}
  \left| \Ie{\calG_h w} (z) - \Ie{w}(z) \right| \lesssim \frac{h}{\varepsilon^2} |\log h| \| w \|_{C^{0,1}(\bar\Omega_\vare)}
\end{equation}
uniformly in $\calA$ and $\calB$.

\subsection{Existence and uniqueness of the numerical solution}
\label{sub:dmp}
%
Based on the notions of \emph{monotonicity} and \emph{weak consistency} advanced in \cite{BarlesSouganidis91, JensenSmears13, RHNWZ}, we show existence and uniqueness of solutions to \eqref{eq:isaacsepsh}. Moreover, the family $\{u_h^\vare\}_{h>0,\vare>0}$ of solutions to \eqref{eq:isaacsepsh} converges to $\ue$, the unique viscosity solution of \eqref{eq:isaacs}, provided the mesh size $h$ and $\vare$ satisfy a suitable relation.

We begin with two elementary properties, namely monotonicity and continuity, of the inf-sup operator.

\begin{lemma}[monotonicity and continuity of inf-sup]
\label{lem:infsup}
Let $\{X^{\alpha,\beta}: \alpha \in \calA, \ \beta \in \calB \}$ and $\{Y^{\alpha,\beta}: \alpha \in \calA, \ \beta \in \calB \}$ be two families parametrized by two index sets $\calA$ and $\calB$. 
If for every fixed $\alpha \in \calA$ and $\beta \in \calB$ we have that
$
  X^{\alpha,\beta} \leq Y^{\alpha,\beta},
$
then
\begin{equation}
\label{eq:infsup}
  \inf_{\alpha \in \calA} \sup_{\beta \in \calB}  X^{\alpha,\beta} \leq \inf_{\alpha \in \calA} \sup_{\beta \in \calB}Y^{\alpha,\beta}.
\end{equation}
Moreover, if there is a constant $C$ such that for all $\alpha$ and $\beta$ we have
$
  \left| X^{\alpha,\beta} - Y^{\alpha,\beta} \right| \leq C,
$
then
\begin{equation}
\label{eq:infsupabs}
  \left| \inf_{\alpha \in \calA} \sup_{\beta \in \calB}  X^{\alpha,\beta} - \inf_{\alpha \in \calA} \sup_{\beta \in \calB}Y^{\alpha,\beta} \right| \leq C.
\end{equation}
\end{lemma}
\begin{proof}
If $ X^{\alpha,\beta} \leq Y^{\alpha,\beta}$ for all fixed $\alpha$ and $\beta$, taking supremum on $Y^{\alpha, \beta}$ with respect to $\beta$ first, we have
\[
  X^{\alpha,\beta} \leq \sup_{\beta \in \calB} Y^{\alpha,\beta}
  \quad \implies \quad
  \sup_{\beta \in \calB} X^{\alpha,\beta} \leq \sup_{\beta \in \calB} Y^{\alpha,\beta}.
\]
Taking infimum on $X^{\alpha, \beta}$ with respect to $\alpha$, we get
\[
  \inf_{\alpha \in \calA}\sup_{\beta \in \calB} X^{\alpha,\beta} \leq \sup_{\beta \in \calB} Y^{\alpha,\beta}
  \quad \implies \quad
  \inf_{\alpha \in \calA}\sup_{\beta \in \calB} X^{\alpha,\beta} \leq \inf_{\alpha \in \calA}\sup_{\beta \in \calB} Y^{\alpha,\beta}.
\]
This proves the first inequality \eqref{eq:infsup}. 

To prove the second inequality, we only need to note that if  $X^{\alpha,\beta} \leq Y^{\alpha,\beta} + C$ for all $\alpha$ and $\beta$, then by the first equality
\[
 \inf_{\alpha \in \calA}\sup_{\beta \in \calB} X^{\alpha,\beta} \leq \inf_{\alpha \in \calA}\sup_{\beta \in \calB} Y^{\alpha,\beta} + C.
\]
Similarly, if $Y^{\alpha,\beta} \leq X^{\alpha,\beta} + C$ for all $\alpha$ and $\beta$, then
\[
 \inf_{\alpha \in \calA}\sup_{\beta \in \calB} Y^{\alpha,\beta} \leq \inf_{\alpha \in \calA}\sup_{\beta \in \calB} X^{\alpha,\beta} + C.
\]
This proves the second inequality \eqref{eq:infsupabs}. 
\end{proof}

We now establish the monotonicity and weak consistency of the proposed scheme.

\begin{lemma}[monotonicity of the numerical scheme]
\label{lem:monotone}
Assume that for every discretization parameter $h>0$ the mesh $\T_h$ is weakly acute. Then the family of operators $\frakF_h^\vare$ is monotone in the sense that if $v_h,w_h \in \polV_h$ with $v_h \leq w_h$ and equality holds at some node $z \in \Nh$, then
\[
  \frakF_h^\vare[v_h](z) \leq \frakF_h^\vare[w_h](z).
\]
\end{lemma}
\begin{proof}
Let $v_h, w_h \in \polV_h$ be as indicated. Since $\T_h$ is weakly acute, the operator $\LAP_h$ satisfies a discrete maximum principle, \ie 
\[
  \LAP_h v_h(z) \leq \LAP_h w_h(z).
\]
It remains then to deal with the inf-sup over the approximating integral operators. 
Since $v_h \leq w_h$ and equality holds at some node $z \in \Nh$, we have $\frakd v_h (z, y) \leq \frakd w_h (z, y)$. Therefore, for fixed $\alpha \in \calA$ and $\beta \in \calB$, we get
$
  \Ie{v_h}(z) \leq \Ie{w_h}(z).
$
Using the monotonicity of the inf-sup operator \eqref{eq:infsup}, we obtain
\[
  \inf_{\alpha \in \calA} \sup_{\beta \in \calB} \Ie{v_h}(z) 
  \leq 
  \inf_{\alpha \in \calA} \sup_{\beta \in \calB} \Ie{w_h}(z).
\]
Combining both inequalities, we then deduce that
\[
  \frakF_h^\vare[v_h](z) \leq \frakF_h^\vare[w_h](z).
\]
This completes the proof.
\end{proof}

The analysis of our scheme is based on a careful estimation of the consistency error, which we define as the difference between the result of applying the operator to the Galerkin projection of the solution to \eqref{eq:isaacseps} and the approximate right hand side, \ie $\Fhe{ \calG_h u^\vare}(z) - f_z$. The critical step in this analysis is to understand what happens with the integral operators. With this in mind, and for future use, for a function $w \in C(\bar\Omega_\vare)$  and $z \in \Nh$ we denote
\begin{equation}
\label{eq:defofcalR}
  \begin{aligned}
    \calR_{h,\vare}[w](z) &= 
    \left( \int_\Omega \phi_z(x) \diff x \right)^{-1}
    \int_\Omega \left[ \inf_{\alpha \in \calA}\sup_{\beta \in \calB} \Ie{\calG_h w}(z) \right. \\
    &= \left. \inf_{\alpha \in \calA}\sup_{\beta \in \calB} \Ie{w}(x) \right] \phi_z(x) \diff x.
  \end{aligned}
\end{equation}

The following result bounds $\calR_{h,\vare}[w](z)$ uniformly in $z \in \Nh$ when $w \in C^{0,1}(\bar\Omega_\vare)$.

\begin{lemma}[Galerkin projection vs.~integral operator]
\label{lem:GvsI}
Let $w \in C^{0,1}(\bar\Omega_\vare)$. If the coefficient matrices are such that \eqref{eq:modulus} holds, then, for every $z \in \Nh$, we have
\begin{equation}
\label{eq:boundcalR}
  \left| \calR_{h,\vare}[w](z) \right| \lesssim \left( \frac{h}{\vare^2}|\log h| + \frac{\varpi(h)}\vare \right) \| w \|_{C^{0,1}(\bar\Omega_\vare)},
\end{equation}
where the hidden constant is independent of $h$, $\vare$ and $w$.
\end{lemma}
\begin{proof}
Notice that, by Lemma~\ref{lem:infsup} it suffices to bound, for $x \in \supp \phi_z$,
\begin{align*}
  r^{\alpha,\beta}_{h,\vare}[w](z) &= \Ie{\calG_h w}(z) - \Ie{w}(x) \\
    &= \left( \Ie{\calG_h w}(z) - \Ie{w}(z) \right) + \left( \Ie{w}(z) - \Ie{w}(x) \right) \\
    &= r_{1,h,\vare}^{\alpha,\beta}[w](z) + r_{2,h,\vare}^{\alpha,\beta}[w](z).
\end{align*}
Estimate \eqref{estimateIe} immediately yields
\[
  |r_{1,h,\vare}^{\alpha,\beta}[w](z) | \lesssim \frac{h}{\vare^2} |\log h| \| w \|_{C^{0,1}(\bar\Omega_\vare)}.
\]
On the other hand, Lemma~\ref{lem:continuous} implies
\[
  |r_{2,h,\vare}^{\alpha,\beta}[w](z) | \lesssim \left( \frac{h}{\vare^2} + \frac{\varpi(h)}\vare \right) \| w \|_{C^{0,1}(\bar\Omega_\vare)},
\]
where we used that $|x-z| \lesssim h$.
\end{proof}

The monotonicity property given in Lemma~\ref{lem:monotone} ensures that, for every $\vare >0$ and $h>0$, the scheme \eqref{eq:isaacsepsh} has a unique solution. The main idea behind the the existence and uniqueness of the numerical solution is a discrete variant of Perron's method \cite[\S6.1]{MR2777537}.

\begin{theorem}[existence and uniqueness]
\label{lem:existenceuniqueness}
Let the family of meshes $\{\T_h\}_{h>0}$ be weakly acute. For every $h>0$ and $\vare>0$ the finite element scheme \eqref{eq:isaacsepsh} has a unique solution.
\end{theorem}
\begin{proof}
We first prove the existence of a solution in several steps.

\noindent \boxed{1} We define the set of discrete super-solutions
\[
  \polS_h =  \left\{ v_h \in \polV_h : \; \Fhe{v_h}(z) \leq f_{z} \  \forall z \in \Nh, \; v_h(z) \geq 0 \ \forall z \in \Nh^\partial \right\}.
\]
We claim that the set $\polS_h$ is nonempty. Set
\[
  \delta_{h,\vare} = C \left( \frac{h}{\vare^2}|\log h| + \frac{\varpi(h)}\vare \right),
\]
where $C>0$ is a constant to be chosen later. Let $F = \min \{0,\min_{\bar\Omega} \fe\}-\delta_{h,\vare}$ and $B_R$ be a ball centered at the origin of radius $R$ that contains $\Omega$. Let $P(x) =  \frac12 \lambda^{-1} F (|x|^2 - R^2)$ be a quadratic polynomial that is nonnegative in $B_R$, $P(x) \geq 0$ on $\partial \Omega$, and $\D^2 P= \lambda^{-1} F I \leq 0$. 
Owing to the uniform ellipticity condition \eqref{eq:ellipticity}, for each $\alpha$ and $\beta$, we obtain
\[
  A^{\alpha,\beta}(x):\D^2 P(x) \leq \lambda ( \lambda^{-1} F) \leq \fe(x) - \delta_{h,\vare}
  \quad \implies \quad
  \frakF[P](x) \leq  \fe(x) - \delta_{h,\vare}.
\]
We now claim that $\calG_h P \in \polS_h$. 
To see this consider
\begin{align*}
  \Fhe{\calG_h P}(z) &= \frac\lambda2 \LAP_h \calG_h P(z) + \inf_{\alpha \in \calA}\sup_{\beta \in \calB} \Ie{\calG_h P}(z) \\
  &= \left(\int_\Omega \phi_z(x) \diff x \right)^{-1} 
    \int_\Omega \left(  \LAP P(x) + \inf_{\alpha \in \calA}\sup_{\beta \in \calB} \Ie{P}(x) \right) \phi_z(x) \diff x \\
  &+ \calR_{h,\vare}[P](z).
\end{align*}
Using the consistency of $\Ie{\cdot}$ for quadratics, see Lemma~\ref{lem:approximation}, we conclude that
\begin{equation}
\label{eq:consQuads}
  \Fhe{\calG_h P}(z) = \frac{  \int_\Omega \frakF[P](x) \phi_z(x) \diff x }{ \int_\Omega \phi_z(x) \diff x } + \calR_{h,\vare}[P](z)
  \leq f_z - \delta_{h,\vare} + \calR_{h,\vare}[P](z).
\end{equation}
Lemma~\ref{lem:GvsI} and the fact that $\| P \|_{C^{0,1}(\bar\Omega_\vare)} \lesssim \| \fe \|_{L^\infty(\Omega)}$, show that we can choose $C>0$ sufficiently large so that $\calR_{h,\vare}[P](z)-\delta_{h,\vare} \leq 0$ for all $z \in \Nh$. Therefore,
\[
  \Fhe{\calG_h P}(z) \leq f_z \implies \calG_h P \in \polS_h.
\]
Thus, $\polS_h$ is nonempty.

\noindent \boxed{3} Given $v_h,w_h \in \polS_h$, we define $(v \wedge w)_h \in \Vh$ by
\[
  (v \wedge w)_h(z) = \min\{ v_h(z), w_h(z) \} \quad \forall z \in \Nh \cup \Nh^\partial.
\]
Since $v_h(z) \geq 0$ and $w_h(z) \geq 0$ on the boundary nodes $z \in \Nh^{\partial}$, we have 
\[
(v \wedge w)_h(z) \geq 0.
\]
Notice that, if there is $z^0 \in \Nh$ such that $(v \wedge w)_h(z^0) = w_h(z^0)$ we get
\[
  (v \wedge w)_h \leq w_h \qquad \text{and} \qquad (v \wedge w)_h(z^0) = w_h(z^0).
\]
The monotonicity result of Lemma~\ref{lem:monotone} implies that for all $z \in \Nh$
\[
  \Fhe {(v \wedge w)_h }(z) \leq \Fhe{w_h} (z) \leq f_{z},
\]
so that $(v \wedge w)_h \in \polS_h$.

\noindent \boxed{4} Let $u_h^\star \in \Vh$ be defined by
\[
  u^\star_h(z) = \inf_{v_h \in \polS_h} v_h(z) \quad \forall z \in \Nh.
\] 
The reasoning given above shows that $u_h^\star  \in \polS_h$. We claim that $u_h^\star$ is a solution, for if that is not the case, then either:

\noindent \textbf{Case I}: There is a boundary node $z\in \Nh^{\partial}$ such that $u_h^\star (z) > 0$. Define $v_h^\star = u_h^\star - \delta \phi_{z}$.  For $\delta>0$ sufficiently small, $v_h^\star (z) \geq 0$. The monotonicity of the scheme shown in Lemma~\ref{lem:monotone} implies that $v_h^\star$ is still a super-solution. This contradicts the fact that $u_h^\star$ is the smallest super-solution of \eqref{eq:isaacsepsh}.

\noindent \textbf{Case II}: There exists an interior node $z \in \Nh$ such that 
\[
 \Fhe {u_h^\star}(z) < f_{z}.
\]
We, again, define $v_h^\star = u_h^\star - \delta \phi_{z}$. For $\delta$ sufficiently small, 
\[
\Fhe {v_h^\star} (z) \leq f_{z}
\] 
and, thanks to Lemma~\ref{lem:monotone}, ${v_h^\star}$ is again a super-solution. In this case, we also obtain a contradiction.

\noindent This proves the existence of a solution.

We now prove uniqueness. If $u_h$ and $v_h$ are two solutions of the finite element scheme \eqref{eq:isaacsepsh}, set $w_h = u_h - v_h$. If $w_h \neq 0$ then, without loss of generality, we may assume that $w_h$ attains a strict maximum at the node $z \in \Nh$. At this node we have
\[
  \frac{\lambda}{2} \LAP_h w_h(z) < 0 \qquad \Ie{w_h} (z) < 0    \quad \forall \alpha \in \calA, \beta \in \calB,
\]
or, equivalently,
\begin{align*}
  \frac{\lambda}{2} \LAP_h{u_h}(z) +  \Ie{u_h}(z) <  \frac{\lambda}{2} \LAP_h{v_h}(z) +  \Ie{v_h}(z) + \frac\lambda2 \LAP_h w_h(z)
\end{align*}
for all $\alpha$ and $\beta$.
By the monotonicity of the inf-sup operator \eqref{eq:infsup}, we have
\[
  \Fhe{u_h}(z) \leq \Fhe{v_h}(z) + \frac\lambda2 \LAP_h w (z) < \Fhe{v_h}(z),
\]
which contradicts the fact that both $u_h$ and $v_h$ solve \eqref{eq:isaacsepsh}.
\end{proof}

\begin{remark}[alternate proof of existence]
\label{rem:otherex}
An alternative proof of existence of solutions will be given in Theorem~\ref{thm:convergenceHoward} via the convergence of Howard's algorithm.
\end{remark}

\subsection{Convergence to the viscosity solution}
\label{sub:convergence}
Combining the discrete maximum principle and consistency estimate with the result in \cite{BarlesSouganidis91, KuoTrudinger92}, we prove convergence of our scheme \eqref{eq:isaacsepsh}. For convenience, let us recall the convergence result of \cite{BarlesSouganidis91, KuoTrudinger92}.

\begin{lemma}[convergence criterion \cite{KuoTrudinger92}]
\label{lem:criterion}
Let $\Omega$ be a bounded domain in $\mathbb{R}^d$ satisfying an exterior cone condition. A consistent and monotone scheme converges uniformly to the viscosity solution \eqref{eq:isaacs}.
\end{lemma}

Our convergence result then follows from Lemma~\ref{lem:criterion}. For brevity we only sketch the proof, as it repeats standard arguments.

\begin{theorem}[convergence]
\label{thm:convergence}
Assume that $\fe \in C^{0,1}(\bar\Omega)$ and that the coefficients are such that \eqref{eq:ellipticity} and \eqref{eq:modulus} hold. Let $u_h^\vare \in \polV_h^0$ be the solution to \eqref{eq:isaacsepsh}. Denote
\[
  u^\star(x) = \limsup_{h\to0,z \to x} u_h^\vare(z), \qquad
  u_\star(x) = \liminf_{h\to0,z \to x} u_h^\vare(z),
\]
for any sequence $z \to x$ and let 
\[
  \vare \gtrsim \max\{ h^{1/2} |\log h|, \varpi(h) \},
\]
and $\vare \to 0$. Then the upper semi-continuous function $u^\star$ is a viscosity subsolution of \eqref{eq:isaacs} and the lower semi-continuous function $u_\star$ is a viscosity supersolution of \eqref{eq:isaacs}.
\end{theorem}
\begin{proof}
We only show that $u^\star$ is a subsolution. Let $w$ be a quadratic polynomial and such that $u^\star - w$ has a local maximum at $x_0 \in \Omega$. We need to show that
\[
  \frakF[w](x_0) \geq \fe(x_0).
\]
Without loss of generality 
we can assume that this is a strict maximum. As in the proof of \cite[Lemma 4.1]{RHNWZ} we can show that if $u_h^\vare - \calG_h w$ attains its maximum at $z \in \Nh$, then we must have $z \to x_0$ as $h \to 0$.

Lemma~\ref{lem:monotone} and the fact that $u_h^\vare - \calG_h w$ attains its maximum at $z$ yield
\[
  \frakF_h^\vare[\calG_h w](z) \geq \frakF_h^\vare[u_h^\vare](z) = f_z.
\]
Since $\fe$ is continuous $f_z \to \fe(x_0)$ and so it remains to study the behavior of the left hand side in this inequality.

Repeating the computations of Step \boxed{1} in the proof of Theorem~\ref{lem:existenceuniqueness} that led to \eqref{eq:consQuads} allows us to obtain
\[
  \Fhe{\calG_h w}(z) = \frac{  \int_\Omega \frakF[w](x) \phi_z(x) \diff x }{ \int_\Omega \phi_z(x) \diff x } + \calR_{h,\vare}[w](z),
\]
where $\calR_{h,\vare}[w](z)$ is defined in \eqref{eq:defofcalR}.
Consequently, if we are able to show that, for every $z \in \Nh$, $\calR_{h,\vare}[w](z) \to 0$ as $h\to 0$, $\vare \to 0$ we would get that
\[
  \fe(x_0) \leq \limsup_{h \to 0} \frakF_h^\vare[\calG_h w](z) \leq \frakF[w](x_0),
\]
which is what we need to prove. The bound on $\calR_{h,\vare}[w](z)$, obtained in Lemma~\ref{lem:GvsI}, and the choice of $\vare$ allow us to conclude.
\end{proof}

Classical arguments involving comparison principles allow us now to show convergence to the unique viscosity solution.

\begin{corollary}[convergence]
\label{cor:convergence}
The family $\{u_h^\vare\}_{h>0,\vare>0}$ of solutions to \eqref{eq:isaacsepsh} converges pointwise to the viscosity solution $\ue$ of \eqref{eq:isaacs} as $\vare, h \to 0$, provided that $\vare \gtrsim \max\{ h^{1/2} |\log h|, \varpi(h) \}$.
\end{corollary}
\begin{proof}
Using the notation of Theorem~\ref{thm:convergence} we note that, by definition $u_\star \leq u^\star$. Moreover, a comparison principle and the fact that $u_\star = u^\star = 0$ on $\partial\Omega$ readily imply that $u^\star = u_\star =\ue$, the viscosity solution of \eqref{eq:isaacs}. This proves convergence of $u_h^\vare$ to $\ue$.
\end{proof}

\section{Rate of convergence}
\label{sec:rates}

Let us now study the rate of convergence of $u_h^\vare$ to $\ue$, under the smoothness assumptions of Corollary~\ref{cor:reg}. We will achieve this by comparing the solution of the discrete scheme \eqref{eq:isaacsepsh} with the solution of the integrodifferential approximation \eqref{eq:isaacseps}. In light of the technique that we are adopting, we must immediately note that since the approximation results of Proposition~\ref{prop:CS} and, as a consequence, the regularity obtained in Corollary~\ref{cor:reg} only apply in the case of constant coefficients, $A^{\alpha,\beta}(x) = A^{\alpha,\beta}$, in what follows we must assume this. In this setting \eqref{eq:modulus} is trivially satisfied with $\varpi \equiv 0$. It is possible that the results of \cite{MR3412399} can be used to extend Proposition~\ref{prop:CS} to variable coefficients and if that is the case, our results will immediately follow as well.

We finally remark that since the relation $\vare \gtrsim h^{1/2} |\log h|$ is required for convergence of our method we will study the rate of convergence under this assumption.

\subsection{The discrete Alexandrov Bakelman Pucci estimate}
\label{sub:ABPh}

We begin by recalling the fundamental discrete ABP estimate of \cite[Lemma 6.1]{RHNWZ}. Let $B_R$ be a ball of radius $R$ which contains the domain $\Omega$. 
We define the convex envelope of a function $v_h \in \Vh^0$
\begin{equation}\label{convexenvelope}
  \Gamma (v_h) (x) := \sup\left\{ L(x) : L \in \polP_1, \ L \leq -v_h^- \text{ in } B_{R} \right\},
\end{equation}
where we denote by $v_h^-$ the negative part of $v_h$ in $\Omega$ and $v_h^- :=0$ in $B_R \setminus \Omega$. We also define the \emph{(lower) nodal contact set}
\begin{align}\label{contactset}
  \calC_h^-(v_h) = \left\{ z \in \Nh : \Gamma(v_h)(z) = v_h(z) \right\}.
\end{align}

\begin{lemma}[discrete ABP estimate]
\label{lem:ABPh}
Let the mesh $\T_h$ be shape regular and such that it allows a discrete maximum principle for the discrete Laplacian. If $v_h \in \polV_h^0$
satisfies
\[
 \LAP_h v_h(z) \leq f_{z} \quad \forall z \in \Nh,
\]
then
\[
  \sup_\Omega v_h^- \lesssim \left( \sum_{z \in \calC^-(v_h) } |f_z|^d |\omega_z| \right)^{1/d},
\]
where the hidden constant is independent of $h$ and $\omega_{z} = \supp \phi_z$.
\end{lemma}

This inequality gives us an estimate on the negative part of $v_h$, while a bound for its positive part can be derived in the same fashion by considering a concave envelope and the corresponding (upper) contact set. Combining these bounds yields that the $L^{\infty}$-norm of $v_h$ is controled by the $L^d_h$-norm of $\{f_z\}_{z \in \Nh}$. The discrete ABP estimate, moreover, will allow us to establish rates of convergence.

\subsection{The error equation}
\label{sub:erroreq}
We now derive an equation to determine the error. Multiply the integro-differential approximation to the Isaacs equation \eqref{eq:isaacseps} by a test function $\phi_{z}$, integrate over $\Omega$ and divide by $\int_\Omega \phi_{z}(x) \diff x$.
In view of the definition of the discrete Laplacian \eqref{discretelaplace} and of the Galerkin projection \eqref{eq:defofGh}, we obtain
\[
  \frac\lambda2 \LAP_h \calG_h u^\vare(z) + \inf_{\alpha \in \calA} \sup_{\beta \in \calB} \Ie{\calG_h u^\vare}(z) = f_{z} + \calR_{h,\vare}[u^\vare](z).
\]
Subtract this equation from the scheme \eqref{eq:isaacsepsh} to obtain, for all $z \in \Nh$,
\begin{equation}
\label{eq:erreq}
  \begin{aligned}
  \frac\lambda2 \LAP_h \left(\calG_h u^\vare - u_h^\vare\right)(z) &+ 
  \inf_{\alpha \in \calA}\sup_{\beta\in\calB} \Ie{\calG_h u^\vare}(z) \\
  &- \inf_{\alpha \in \calA}\sup_{\beta\in\calB} \Ie{u^\vare_h}(z) =
  \calR_{h,\vare}[u^\vare](z).
  \end{aligned}
\end{equation}

\subsection{Rate of convergence}
\label{sub:rate}

With the error equation \eqref{eq:erreq} at hand, we now readily obtain an error estimate. This is the content of the next result.

\begin{theorem}[rate of convergence]
\label{thm:rate}
Let $\ue \in C^{1,s}(\Omega)\cap C^{0,1}(\bar\Omega)$ be the viscosity solution of the Isaacs equation \eqref{eq:isaacs} and $u_h^\vare \in \polV_h^0$ be the solution of \eqref{eq:isaacsepsh}. Choose $\vare \gtrsim h^{1/2}|\log h|$. Then, there is $\sigma>0$ such that
\[
  \| \ue - u_h^\vare \|_{L^\infty(\Omega)} \lesssim \left( \vare^\sigma + \frac{h}{\vare^2} |\log h| \right)\| \fe \|_{C^{0,1}(\bar\Omega)},
\]
where the hidden constant depends on $\lambda$ and $\Lambda$, but is independent of $\vare$ and $h$.
\end{theorem}
\begin{proof}
We write
\[
  \| \ue - u_h^\vare \|_{L^\infty(\Omega)} \leq \| \ue - u^\vare \|_{L^\infty(\Omega)} + \| u^\vare - \calG_h u^\vare \|_{L^\infty(\Omega)} 
  + \| \calG_h u^\vare - u_h^\vare \|_{L^\infty(\Omega)},
\]
and we examine each term separately.

\noindent \boxed{1} By Proposition~\ref{prop:CS}, there is $\sigma>0$ such that
\[
  \| \ue - u^\vare \|_{L^\infty(\Omega)} \lesssim \vare^\sigma \| \fe \|_{C^{0,1}(\bar\Omega)}.
\]
  
\noindent \boxed{2} Estimate \eqref{eq:Galproj} on the Galerkin projection immediately yields
\[
   \| u^\vare - \calG_h u^\vare \|_{L^\infty(\Omega)} \lesssim h |\log h| \| \fe \|_{C^{0,1}(\bar\Omega)},
\]
where, using Proposition~\ref{prop:CS}, we bounded $\| u^\vare \|_{C^{0,1}(\bar\Omega)}$ by $\| \fe \|_{C^{0,1}(\bar\Omega)}$.

\noindent \boxed{3} Let us denote $e_h = \calG_h u^\vare - u_h^\vare \in \polV_h^0$. The definition of the convex envelope \eqref{convexenvelope} implies that
\[
  e_h(x) = \calG_hu^\vare(x) - u_h^\vare(x) \geq \Gamma (e_h) (x) \quad \forall x \in \Omega,
\]
and that if $z \in \calC_h^-(e_h)$ then $e_h(z) = \Gamma(e_h)(z)$.
Since $\Gamma (e_h)(x)$ is convex, we obtain $\frakd \calG_hu^\vare(z,y) \geq \frakd u_h^\vare(z,y)$. Consequently, for every $\alpha \in \calA$ and $\beta \in \calB$,
\[
  \Ie{\calG_h u^\vare}(z) \geq \Ie{ u_h^\vare}(z), \quad \forall z \in \calC_h^-(e_h),
\]
so that using the monotonicity property of the inf-sup operator \eqref{eq:infsup}, we have
\[
  \inf_{\alpha \in \calA}\sup_{\beta \in \calB} I^{\alpha,\beta}_\vare \calG_h u^\vare(z) \geq 
  \inf_{\alpha \in \calA}\sup_{\beta \in \calB} I^{\alpha,\beta}_\vare u_h^\vare(z),
  \quad \forall z \in \calC_h^-(e_h).
\]
In conclusion, at the nodal contact set $\calC^-(e_h)$ equation \eqref{eq:erreq} reduces to
\[
\frac\lambda2 \LAP_h e_h(z) \leq \calR_{h,\vare}[u^\vare](z).
\]
An application of the discrete ABP estimate of Lemma~\ref{lem:ABPh} then yields
\[
  \sup_\Omega e_h^- \lesssim \left( \sum_{z \in \calC^-(e_h)} |\calR_{h,\vare}[u^\vare](z)|^d |\omega_z| \right)^{1/d} \lesssim 
  \max_{z \in \Nh} |\calR_{h,\vare}[u^\vare](z)|.
\]
This yields a lower bound for $e_h$. An upper bound can be derived in a similar fashion by considering $- e_h$. Hence, we infer
\[
  \| \calG_h u^\vare - u_h^\vare \|_{L^{\infty}(\Omega)}  \lesssim \max_{z \in \Nh} |\calR_{h,\vare}[u^\vare](z)|.
\]
Using the fact that $u^\vare \in C^{0,1}(\bar\Omega)$ uniformly in $\vare$ we can invoke the bounds on $\calR_{h,\vare}[u^\vare](z)$ obtained in Lemma~\ref{lem:GvsI} which, recalling that $\varpi \equiv 0$, can be combined with the estimate of Proposition~\ref{prop:CS} to yield
\[
  \max_{z \in \Nh} |\calR_{h,\vare}[u^\vare](z)| \lesssim \frac{h}{\vare^2} |\log h| \| \fe \|_{C^{0,1}(\bar\Omega)}.
\]
Notice now that the choice of $\vare$ implies that $h|\log h|\vare^{-2} \to 0$ as $h \to 0$ and $\vare \to 0$. Therefore, eventually, we obtain
\[
  \| \calG_h u^\vare - u_h^\vare \|_{L^{\infty}(\Omega)}  \lesssim \frac{h}{\vare^2} |\log h| \| \fe \|_{C^{0,1}(\bar\Omega)} .
\]

Combining the estimates of these three steps yields the result.
\end{proof}

\begin{remark}[choice of $\vare$]
If one knows the value of $\sigma$ in Proposition~\ref{prop:CS}, setting $\vare^{\sigma+2} = h|\log h|$ yields an error estimate of the form 
\[
  \| \ue - u_h^\vare \|_{L^\infty(\Omega)} \lesssim h^{\frac\sigma{\sigma+2}} |\log h |^{\frac\sigma{\sigma+2}} \| \fe \|_{C^{0,1}(\bar\Omega)}.
\]
Notice that in the best case scenario, that is $\sigma = 1$, we would obtain a rate of convergence of order $\calO(h^{1/3}|\log h|^{1/3})$.
\end{remark}

\begin{remark}[explicit rate of convergence]
The rate of convergence given in Theorem~\ref{thm:rate} is given rather implicitly. It seems that this is a recurring feature in the literature; see for instance the main result in \cite{MR3355497}.
\end{remark}

\section{Implementation details}
\label{sec:implementation}
Let us discuss some details on how to actually implement and solve the nonlinear problem that scheme \eqref{eq:isaacsepsh} entails. We will first discuss how to obtain a solution to the nonlinear problem and then show how, through the use of quadrature, we can modify the scheme to make it amenable to implementation while not losing any of the properties that we have obtained in previous sections.

\subsection{The solution scheme}
\label{sub:solscheme}
The main difficulty in devising a convergent algorithm for the solution of \eqref{eq:isaacsepsh} is the fact that, due to the inf-sup operations, the underlying operator is neither convex nor concave. This is in sharp contrast with, for instance, the interpretation of Howard's algorithm \cite{MR0177813} as a semi-smooth Newton method described in \cite{JensenSmears13,MR3196952} for the Hamilton Jacobi Bellman since, as it is shown in \cite[Remark 5.3]{MR2551155} such a method may not converge.

On the other hand, \cite[Section 5]{MR2551155} presents a convergent generalization of Howard's algorithm for max-min problems, which we readily adapt here. We will present an algorithm which requires the solution, at every iteration step, of a Hamilton Jacobi Bellman equation, which can be realized via a semi-smooth Newton method. We will also comment on an algorithm with inexact solves.

For a given $w_h \in \polV_h^0$ and $z \in \N_h$ define $\alpha(w_h,z) \in \calA$ as the element that infimizes the supremum of the integral operators when applied to $w_h$ at the point $z$, that is
\[
  \inf_{\alpha \in \calA} \sup_{\beta \in \calB} I_{\vare}^{\alpha, \beta}[w_h](z) = \sup_{\beta \in \calB} I_{\vare}^{\alpha(w_h,z), \beta}[w_h](z).
\]
Set $\balpha(w_h) = \{\alpha(w_h,z): z \in \N_h \}$. For $v_h \in \polV_h^0$ define the operator $\frakF_{h,\vare}^{\balpha(v_h)}: \polV_h^0 \to \polV_h^0$ by
\begin{equation}
\label{eq:defofFa}
  \frakF_{h,\vare}^{\balpha(v_h)}[w_h](z) := \frac\lambda2 \LAP_h w_h(z) + \sup_{\beta \in \calB} I_\vare^{\alpha(v_h,z),\beta}[w_h](z).
\end{equation}
Our algorithm can then be described as follows:
\begin{enumerate}[$\bullet$]
  \item \textbf{Initialization}: Choose $w_h^{-1} \in \polV_h^0$ and set $\balpha_0 = \balpha(w_h^{-1})$.
  
  \item \textbf{Iteration}: For $k \geq 0$ find $w_h^k \in \polV_h^0$ that solves
  \begin{equation}
  \label{eq:wk}
    \frakF_{h,\vare}^{\balpha_k}[w_h^k](z) = f_{z}, \quad \forall z \in \N_h
  \end{equation}
  and set
  \begin{equation}
  \label{eq:balphak}
    \balpha_{k+1} = \balpha(w_h^k).
  \end{equation}
  
  \item \textbf{Convergence test}: If $\frakF_{h,\vare}[w_h^k](z) = f_{z}$ for all $z \in \N_h$ stop.
\end{enumerate}

Notice that \eqref{eq:balphak} is equivalent to
\[
  \frakF_{h,\vare}[w_h^k] = \frakF_{h,\vare}^{\balpha_{k+1}}[w_h^k].
\]
The analysis of the algorithm \eqref{eq:wk}--\eqref{eq:balphak} relies on the following properties of the operators $\frakF_{h,\vare}^\balpha$.

\begin{lemma}[monotonicity and comparison]
\label{lem:compare}
For every $\balpha \in \calA^{\#\N_h}$ the operator $\frakF_{h,\vare}^\balpha$ is monotone and satisfies a comparison principle, \ie if for $v_h, w_h \in \polV_h^0$ we have
\[
  \frakF_{h,\vare}^\balpha[v_h] \leq \frakF_{h,\vare}^\balpha[w_h],
\]
then $v_h \geq w_h$.
\end{lemma}
\begin{proof}
Monotonicity can be established as in Lemma~\ref{lem:monotone}. Let us now establish the comparison principle. Define $e_h = v_h - w_h$ and assume that it has a strict minimum on $z \in \Nh$. This implies
\[
  \frac\lambda2 \LAP_h e_h (z) > 0, \qquad I_\vare^{\alpha,\beta}[e_h](z) > 0,
\]
for every $\alpha \in \calA$, $\beta \in \calB$. Consequently,
\[
  \frac\lambda2 \LAP_h e_h(z) + \frac\lambda2 \LAP_h w_h(z) + I_\vare^{\alpha,\beta}[w_h] (z) < 
  \frac\lambda2 \LAP_h v_h(z) + I_\vare^{\alpha,\beta}[v_h](z),
\]
this readily implies that
\[
  \frac\lambda2 \LAP_h e_h(z) + \frac\lambda2 \LAP_h w_h(z) + \sup_{\beta \in \calB} I_\vare^{\alpha,\beta}[w_h](z) \leq
  \frac\lambda2 \LAP_h v_h(z) + \sup_{\beta \in \calB} I_\vare^{\alpha,\beta}[v_h](z)
\]
or that, at $z$,
\[
  \frakF_{h,\vare}^\balpha[w_h](z) < \frakF_{h,\vare}^\balpha[v_h](z).
\]

In conclusion, $e_h$ cannot have a minimum in the interior, but since $e_h = 0$ on $\partial\Omega$, we must have that $v_h \geq w_h$.
\end{proof}

The comparison principle will allow us to obtain convergence.

\begin{theorem}[convergence]
\label{thm:convergenceHoward}
The sequence $\{w_h^k\}_{k\geq0} \subset \polV_h^0$ obtained by algorithm \eqref{eq:wk}--\eqref{eq:balphak} converges in a finite number of steps to $u_h^\vare \in \polV_h^0$, solution of \eqref{eq:isaacsepsh}.
\end{theorem}
\begin{proof}
We proceed in two steps:
\begin{enumerate}[1.]
  \item We show that the sequence is monotone, \ie for every $z \in \N_h$, $w_h^k(z) \geq w_h^{k+1}(z)$. By construction:
  \[
    \frakF_{h,\vare}^{\balpha_{k+1}}[w_h^k](z) = \frakF_{h,\vare}[w_h^k](z) \leq \frakF_{h,\vare}^{\balpha_k}[w_h^k](z).
  \]
  Subtracting $f_{z}$ from this inequality we realize that
  \[
    \frakF_{h,\vare}^{\balpha_{k+1}}[w_h^k](z) - f_{z} \leq \frakF_{h,\vare}^{\balpha_k}[w_h^k](z) - f_{z} = 0 = 
    \frakF_{h,\vare}^{\balpha_{k+1}}[w_h^{k+1}](z) - f_{z},
  \]
  which by the comparison principle established in Lemma~\ref{lem:compare} implies $w_h^k \geq w_h^{k+1}$.
  
  \item Since the set $\calA$ is finite, there are at most $(\# \calA)^{\# \N_h}$ different variables and there must be two indices $\kappa> \ell$ for which \eqref{eq:balphak} yields $\balpha_\kappa = \balpha_\ell$. This implies that $w_h^\kappa = w_h^\ell$ and by monotonicity $w_h^k = w_h^\ell$ for all $k\geq l$. But then, by uniqueness $w_h^\ell = u_h^\vare$.
  
\end{enumerate}
This concludes the proof.
\end{proof}

Notice that \eqref{eq:wk} requires the exact solution of a discrete version of a Hamilton Jacobi Bellman problem, which can be achieved by Howard's algorithm. We also propose a scheme with inexact solves in this step:
\begin{enumerate}[$\bullet$]
  \item \textbf{Initialization}: Choose $w_h^{-1} \in \polV_h^0$ and set $\balpha_0 = \balpha(w_h^{-1})$.
  
  \item \textbf{Iteration}: For $k \geq 0$ find $w_h^k \in \polV_h^0$ such that
  \begin{equation}
  \label{eq:wk2}
    \max_{z \in \N_h} \left | \frakF_{h,\vare}^{\balpha_k}[w_h^k](z) - f_{z} \right | < \eta_k
  \end{equation}
  and set
  \begin{equation}
  \label{eq:balphak2}
    \balpha_{k+1} = \balpha(w_h^k).
  \end{equation}
  
  \item \textbf{Convergence test}: If $\frakF_{h,\vare}[w_h^k](z) = f_{z}$ for all $z \in \N_h$ stop.
\end{enumerate}

The convergence of this algorithm follows \emph{mutatis mutandis} the proof of Theorem~5.4 of \cite{MR2551155}.

\begin{theorem}[convergence with inexact solves]
\label{thm:conv}
Assume that the sequence of errors $\{\eta_k\}_{k \in \polN} \in \ell^1$. Then the sequence $\{w_h^k\}_{k\geq0} \subset \polV_h^0$, obtained by algorithm \eqref{eq:wk2}--\eqref{eq:balphak2}, converges to $u_h^\vare \in \polV_h^0$, solution of \eqref{eq:isaacsepsh}.
\end{theorem}

\subsection{A numerical scheme with quadrature}
\label{sub:quadrature}

One of the most difficult aspects in the implementation of the scheme developed in \cite{RHNWZ} and, as a consequence, of that developed here is the computation of the integral operators. In an attempt to simplify this, here we consider a scheme with quadrature. To focus on the essential difficulties, we consider in what follows the linear problem
\begin{equation}
\label{eq:lineq}
  \frakL[\ve](x) := A(x):\D^2 \ve(x) = \fe(x) \ \text{in } \Omega, \quad
  \ve = 0 \ \text{on } \partial\Omega,
\end{equation}
with $A \in C^{0,1}(\bar\Omega,\GL_d(\Real))$ symmetric and elliptic in the sense of \eqref{eq:ellipticity}. Notice that if $\# \calA = \# \calB = 1$ problem \eqref{eq:isaacs} reduces to this one. For this reason, as proposed and analyzed in \cite{RHNWZ}, this problem can be approximated by the integro-differential one
\begin{equation}
\label{eq:lineqeps}
  \frakL^\vare[v^\vare](x) := \frac\lambda2 \LAP v^\vare(x) + I_\vare [v^\vare](x) = \fe(x) \ \text{in } \Omega, \quad
  v^\vare = 0 \ \text{on } \partial\Omega,
\end{equation}
where $I_\vare$ is defined as in \eqref{eq:defofI}. In turn, this problem can be discretized by
\begin{equation}
\label{eq:lineqepsh}
  \frakL^\vare_h[v_h^\vare](z) := \frac\lambda2 \LAP_h v_h^\vare(z) + I_\vare [v_h^\vare](z) = f_{z} \quad \forall z \in \Nh.
\end{equation}
The properties of the scheme \eqref{eq:lineqepsh} have been established in \cite{RHNWZ}. Notice, however, that its implementation requires the exact evaluation of the integral operator at each interior node, which can be very costly. For this reason here we instead consider a scheme with quadrature.

Assume that we have at hand a quadrature formula for the unit ball $B_1$
\[
  \int_{B_1} w(x) \diff x \approx \sum_{j=1}^\nu \omega_j w(\xi_j)
\]
for some nodes $\{\xi_j\}_{j=1}^\nu \subset B_1$ and weights $\{\omega_j\}_{j=1}^\nu \subset \Real^+$ that is exact for $\polP_m$, polynomials of degree $m$. The construction of quadrature formulas of this form which are of maximal degree of precision can be found in classical references like \cite{MR656522}, see also \cite{MR1642507,MR1753510,MR2028042}.

Now, since $\varphi$ is symmetric and compactly supported in $B_1$, applying the change of variables $y = \vare M(x) \zeta$ to \eqref{eq:defofI} yields
\[
  I_\vare[ w](x) = \frac2{\vare^2} \int_{B_1} \left( w(x+\vare M(x) \zeta) - w(x) \right) \varphi(\zeta) \diff \zeta
\]
to which we can apply the quadrature formula and obtain
\begin{equation}
\label{eq:defofIh}
  I_{\vare,h} [w](x) = \frac2{\vare^2} \sum_{j=1}^\nu \omega_j w(x + \vare M(x) \xi_j ) \varphi(\xi_j) - \frac{\kappa_d}{\vare^2} w(x),
\end{equation}
for a constant $\kappa_d$ that depends only on the dimension. With this approximation at hand we define the following scheme
\begin{equation}
\label{eq:lineqepshquad}
  \bar{\frakL}^\vare_h[\tilde{v}_h^\vare](z) := \frac\lambda2 \LAP_h \tilde{v}_h^\vare(z) + I_{\vare,h}[ \tilde{v}_h^\vare](z) = f_{z} \quad 
  \forall z \in \Nh.
\end{equation}

\begin{remark}[implementation]
\label{rem:implem}
Notice that expanding $\tilde{v}_h^\vare = \sum_{y \in \Nh} V_{y} \phi_{y}$ in the nodal basis transforms \eqref{eq:lineqepshquad} into
\[
  \sum_{y \in \Nh} V_{y} \left( \frac\lambda2 \LAP_h \phi_{y}(z) + I_{\vare,h}[\phi_{y}](z) \right) = f_{z}
\]
so that to obtain the system matrix one only needs to apply the quadrature formula \eqref{eq:defofIh} to the nodal basis functions. The definition of $I_{\vare,h}[\cdot]$ yields
\[
  I_{\vare,h}[\phi_{y}](z) = \frac2{\vare^2} \sum_{j=1}^\nu \omega_j \varphi(\xi_j) \phi_{y}(z + \vare M(z) \xi_j)
  - \frac{\kappa_d}{\vare^2} \delta_{z,y}.
\]
The terms $\omega_j \varphi(\xi_j)$ are independent of $z$ and $y$ and thus can be precomputed. On the other hand, we only need to evaluate the basis function $\phi_{y}$ at the points such that
\[
  z + \vare M(z) \xi_j \in \supp \phi_{y}.
\]
\end{remark}

We now discuss the effect of quadrature in the convergence properties of scheme \eqref{eq:lineqepshquad}. A useful characterization of smoothness is that $w \in C^{2,s}(\bar{B}_t)$ whenever there exists a paraboloid $P \in \polP_2$ such that
\[
  \| w - P \|_{L^\infty(\bar{B}_t)} \lesssim t^{2+s}.
\]
This immediately yields that, provided the quadrature formula is exact for quadratics, we have for $w \in C^{2,s}(\bar\Omega)$
\begin{equation}
\label{eq:convquad}
  \| I_\vare w - I_{\vare,h} w \|_{L^\infty(\Omega)} \lesssim \vare^s,
\end{equation}
because we approximate the function locally by a quadratic over the domain of integration, which for a point $x\in \Omega$ is contained in $x+B_\vare$. These properties allow us to establish convergence and, under smoothness assumptions, rates. For brevity we only present one result below.

\begin{theorem}[convergence of the scheme with quadrature]
\label{thm:convquad}
Assume that $\ve$, solution of \eqref{eq:lineq} is such that $\ve \in C^{2,s}(\bar\Omega)$. If $\tilde{v}_h^\vare \in \polV_h^0$ solves \eqref{eq:lineqepshquad}, then
\[
  \| \ve - \tilde{v}_h^\vare \|_{L^\infty(\Omega)} \lesssim \vare^s + \frac{h}{\vare^2} |\log h|,
\]
where the hidden constant is independent of $\vare$ and $h$.
\end{theorem}
\begin{proof}
As it was shown in \cite{MR2667633,RHNWZ}, if $\ve \in C^{2,s}(\bar\Omega)$ the same holds for $v^\vare$, solution of \eqref{eq:lineqeps} and, more importantly,
\[
  \| \ve - v^\vare \|_{L^\infty(\Omega)} \lesssim \vare^s.
\]

It remains then to compare $\tilde{v}_h^\vare$ and $v^\vare$. This can be accomplished by following the arguments we presented in Theorem~\ref{thm:rate},  combined with estimate \eqref{eq:convquad} and the techniques presented in \cite{RHNWZ}.
\end{proof}

\bibliographystyle{plain}
\bibliography{biblio}

\end{document}